\documentclass[a4paper,UKenglish,cleveref, autoref, thm-restate]{lipics-v2021}

\pdfoutput=1 
\hideLIPIcs  

\usepackage{multirow}

\usepackage[ruled,vlined,linesnumbered]{algorithm2e}

\usepackage{xcolor}
\definecolor{mypink}{rgb}{0.9, 0.0, 0.4}
\definecolor{mygreen}{rgb}{0.0,0.5,0.0}
\definecolor{mypurple}{rgb}{0.6,0.0,0.6}
\definecolor{myfuchsia}{rgb}{0.9,0.0,0.9}
\definecolor{mybrown}{rgb}{0.7,0.3,0.3}
\definecolor{mygray}{rgb}{0.6,0.6,0.6}
\definecolor{myltgray}{rgb}{0.85,0.85,0.85}

\definecolor{stnavy}{HTML}{0D1CC9}

\theoremstyle{definition}
\newtheorem{problem}{Problem}
\newcommand{\indeg}{\mathrm{indeg}}
\newcommand{\outdeg}{\mathrm{outdeg}}
\newcommand{\head}{\mathrm{head}}
\newcommand{\tail}{\mathrm{tail}}

\title{Which Phylogenetic Networks are Level-$k$ Networks with Additional Arcs? Structure and Algorithms} 
\titlerunning{Which Phylogenetic Networks are Level-$k$ Networks with Additional Arcs?} 

\author{Takatora Suzuki}{Department of Pure and Applied Mathematics, Graduate School of Fundamental Science and Engineering, Waseda University, Japan}{takatora.szk@fuji.waseda.jp}{https://orcid.org/0009-0007-3425-7006}{JST SPRING Program Grant Number JPMJSP2128, Japan}

\author{Momoko Hayamizu\footnote{Corresponding author}}{Department of Applied Mathematics, Faculty of Science and Engineering, Waseda University, Japan}{hayamizu@waseda.jp}{https://orcid.org/0000-0001-8825-6331}{JST FOREST Program Grant Number
JPMJFR2135, Japan}
\authorrunning{Takatora Suzuki and Momoko Hayamizu} 

\Copyright{Takatora Suzuki and Momoko Hayamizu} 



\category{} 

\relatedversion{} 

\supplement{}
\supplementdetails[subcategory={Source Code, Datasets and Detailed Experimental Results}, cite={}, swhid={}]{Software}{https://github.com/hayamizu-lab/structure-support-networks}

\acknowledgements{The authors thank Tsuyoshi Urata, Haruki Miyaji, and Yukihiro Murakami for their discussions and Manato Yokoyama for help with the GitHub repository.}

\nolinenumbers 

\EventEditors{Bro\v{n}a Brejov\'a  and Rob Patro}
\EventNoEds{2}
\EventLongTitle{25th International Workshop on Algorithms in Bioinformatics (WABI 2025)}
\EventShortTitle{WABI 2025}
\EventAcronym{WABI}
\EventYear{2025}
\EventDate{August 20--22, 2025}
\EventLocation{University of Maryland, in College Park, MD}
\EventLogo{}
\SeriesVolume{xxx}
\ArticleNo{x}

\begin{document}

\maketitle


\begin{abstract}
Reticulate evolution gives rise to complex phylogenetic networks, making their interpretation challenging. A typical approach is to extract trees within such networks. Since Francis and Steel's seminal paper, ``Which Phylogenetic Networks are Merely Trees with Additional Arcs?'' (2015), tree-based phylogenetic networks and their support trees (spanning trees with the same root and leaf set as a given network) have been extensively studied. However, not all phylogenetic networks are tree-based, and for the study of reticulate evolution, it is often more biologically relevant to identify support networks rather than trees. This study generalizes Hayamizu's structure theorem for rooted binary phylogenetic networks, which yielded optimal algorithms for various computational problems on support trees, to extend the theoretical framework for support trees to support networks. This allows us to obtain a direct-product characterization of each of three sets: all, minimal, and minimum support networks, for a given network. Each characterization yields optimal algorithms for counting and generating the support networks of each type. Applications include a linear-time algorithm for finding a support network with the fewest reticulations (i.e., the minimum tier). We also provide exact and heuristic algorithms for finding a support network with the minimum level, both running in exponential time but practical across a reasonably wide range of reticulation numbers.
\end{abstract}

\section{Introduction}
Evolutionary histories involving reticulate events, such as horizontal gene transfer, hybridization, and recombination, are better represented by networks than trees. However, networks can be far more complex than trees, making their interpretation challenging. A typical approach is to find meaningful subgraphs, particularly spanning trees, of these networks.

Since Francis and Steel's seminal paper ``Which Phylogenetic Networks are Merely Trees with Additional Arcs?'' \cite{FrancisSteel-2015-WhichPhylogeneticNetworks}, tree-based networks and their support trees (spanning trees with the same root and leaf set as the network) have been extensively studied (e.g., \cite{Hayamizu-2021-StructureTheoremRooted, HayamizuMakino-2023-RankingTopkTrees, Zhang-2016-TreeBasedPhylogeneticNetworks}). Hayamizu \cite{Hayamizu-2021-StructureTheoremRooted} developed optimal (linear-time or linear-delay)  algorithms for various computational problems, including counting, listing, and optimization of support trees, by establishing a theoretical foundation called a structure theorem for rooted binary phylogenetic networks. It provides a canonical way to decompose any rooted binary phylogenetic network into its unique maximal zig-zag trails, and characterizes the family of edge-sets of support trees of a given network. 
 
While tree-based phylogenetic networks encompass many well-studied subclasses such as tree-child, stack-free, and orchard networks, networks inferred from biological data are not necessarily tree-based. This has led to increased interest in non-tree-based networks (e.g., \cite{DavidovEtAl-2021-MaximumCoveringSubtrees, FischerFrancis-2020-HowTreebasedMy, FrancisEtAl-2018-NewCharacterisationsTreebaseda, SuzukiEtAl--BridgingDeviationIndices}). Indeed, for the study of reticulate evolution, it is often more biologically relevant to consider networks within networks rather than trees. In such contexts, finding a most concise subgraph, such as one with the fewest reticulations (minimum tier) or with the minimum level, is particularly meaningful. These metrics can be interpreted as measures of deviation from being tree-based \cite{FrancisEtAl-2018-NewCharacterisationsTreebaseda, FischerFrancis-2020-HowTreebasedMy}, as they equal zero for tree-based networks.

In this paper, we generalize Hayamizu's structure theorem to build a theoretical framework for support networks of rooted binary phylogenetic networks. Building on the maximal zig-zag trail decomposition established in \cite{Hayamizu-2021-StructureTheoremRooted}, we derive direct-product characterizations of three families of support networks for a given network $N$: all support networks $\mathcal{A}_N$, minimal support networks $\mathcal{B}_N$, and minimum support networks $\mathcal{C}_N$. These characterizations yield closed-form product formulas for their cardinalities, revealing unexpected connections to Fibonacci, Lucas, Padovan, and Perrin numbers. We present a linear-time algorithm for counting each of $|\mathcal{A}_N|$, $|\mathcal{B}_N|$, and $|\mathcal{C}_N|$, as well as a linear-delay algorithm to list the support networks in each family.

Our theoretical results lead to practical algorithms for two key optimization problems: Problem \ref{prob:minimum reticulation} (\textsc{Reticulation Minimization}) and Problem \ref{prob:level opt} (\textsc{Level Minimization}). Problem \ref{prob:minimum reticulation} asks for a support network with the minimum reticulation number (minimum tier). We present a linear-time optimal algorithm for solving Problem \ref{prob:minimum reticulation}, which can be achieved by selecting any element from $\mathcal{C}_N$. Problem \ref{prob:level opt} asks for a support network with the minimum level. We conjecture that this problem is NP-hard and present exact and heuristic  algorithms for solving Problem \ref{prob:level opt} (Algorithms \ref{alg:ExactAlgorithm} and \ref{alg:HeuristicAlgorithm}). Although they are both exponential-time algorithms, they are practical for networks with a reasonably wide range of reticulation numbers. The heuristic is particularly scalable and accurate in most cases.

The remainder of the paper is organized as follows. In Section \ref{sec:preliminaries.def}, we define graph theoretical terminology and review the relevant materials on support trees (Section \ref{subsec:support.trees}) and the structure theorem (Section \ref{sec:preliminaries1}). In Section \ref{sec:Problem}, we  define support networks, their families $\mathcal{A}_N$, $\mathcal{B}_N$, and $\mathcal{C}_N$, and relevant concepts. Section \ref{sec:count.main} provides characterizations and counting formulas for each family. Numerical results are also presented in Section \ref{sec:experiment.case.study}. In Section \ref{sec:minimum.reticulation}, we give a linear-time algorithm for solving Problem \ref{prob:minimum reticulation}. Section \ref{sec:minimum.level} presents the exact and heuristic  algorithms for  Problem \ref{prob:level opt}, along with performance evaluations using synthetic data. Section \ref{sec:conclusion} concludes with a summary of our contributions and directions for future research.

\section{Preliminaries}\label{sec:preliminaries.def}
The graphs in this paper are finite, simple (i.e. having neither loops nor multiple edges), acyclic directed  graphs, unless otherwise stated. 
For a graph $G$, $V(G)$ and $E(G)$ denote the sets of vertices and edges of $G$, respectively. 
For two graphs $G$ and $H$, $G$ is a \emph{subgraph} of $H$ if both $V(G) \subseteq V(H)$ and $E(G) \subseteq E(H)$ hold, in which case we write $G \subseteq H$. 
In particular, $G$ and $H$ are \emph{isomorphic}, denoted by $G=H$, if there exists a bijection between $V(G)$ and $V(H)$ that yields a bijection between $E(G)$ and $E(H)$. 
For a graph $G$, a \emph{block} of $G$ is a maximal connected subgraph that contains no cut edge.
A subgraph $G$ of $H$ is \emph{proper} if $G\neq H$. A subgraph  $G$ of $H$ is a \emph{spanning} subgraph of $H$ if $V(G)=V(H)$.
Given a graph $G$ and a non-empty subset $S \subseteq E(G)$, the edge-set $S$ is said to \emph{induce the subgraph  $G[S]$ of $G$}, that is, the one whose edge-set is $S$ and whose vertex-set is the set of the ends of all edges in $S$. 
For a graph $G$ with $|E(G)| \geq 1$ and a partition  $\{E_1,\dots,E_{d}\}$ of $E(G)$, the collection  $\{G[E_1], \dots, G[E_{d}]\}$ is a  \emph{decomposition} of $G$, and $G$ is said to be \emph{decomposed into} $G[E_1], \dots, G[E_{d}]$. 

For an edge $e=(u,v)$ of a graph $G$, $u$ and $v$ 
are denoted by $\it tail(e)$ and $\it head(e)$, respectively.
For a vertex $v$ of a graph $G$,  the \emph{in-degree of $v$ in $G$}, denoted by $\indeg_G(v)$, is the cardinality of  $\{e\in E(G)\mid \it head(e)=v\}$. The \emph{out-degree of $v$ in $G$}, denoted by $\outdeg_G(v)$,  is defined in a similar way.  For any graph $G$, a vertex $v\in V(G)$ with $\outdeg_G(v)=0$ is called a \emph{leaf} of $G$. \emph{Subdividing} an edge $(u,v)$ means replacing it with a directed path from $u$ to $v$ of length at least one. \emph{Smoothing} a vertex $v$ where $\indeg_G(v)=\outdeg_G(v)=1$ means suppressing $v$ from $G$, namely, the reverse operation of edge subdivision.

\subsection{Phylogenetic networks and support trees}\label{subsec:support.trees}
Suppose $X$ represents a non-empty finite set of present-day species. A \emph{rooted binary phylogenetic network (on a leaf set $X$)} is defined to be a finite simple directed acyclic graph $N$ with the following three properties:
\begin{description}
  \item[(P1)] There exists a unique vertex $\rho$ of $N$ with $\indeg_N(\rho)=0$ and $\outdeg_N(\rho)\in \{1,2\}$; 
  \item[(P2)] The set of leaves of $N$ is identical to $X$; 
  \item[(P3)] For any $v\in V(N)\setminus (X\cup \{\rho\})$, $\indeg_N(v) \in \{1,2\}$ and $\outdeg_N(v) \in \{1,2\}$.
\end{description}

The vertex $\rho$ is called \emph{the root} of $N$, and any vertex  $v$ with $(\indeg_N(v), \outdeg_N(v))= (2,1)$ is called a \emph{reticulation} of $N$. 
When $N$ has no reticulation, $N$ is particularly called a rooted binary phylogenetic \emph{tree} on $X$.

\begin{definition}[\cite{FrancisSteel-2015-WhichPhylogeneticNetworks}]\label{def:support tree}
Let $N$ be a rooted binary phylogenetic network on $X$. If there exist  a rooted binary phylogenetic tree $\tilde{T}$ on $X$ and  a spanning tree $T$ of $N$  such that $\tilde{T}$ is obtained by smoothing all $v\in V(T)$ with $\indeg_T(v)=\outdeg_T(v)=1$, then $N$ is called a \emph{tree-based} network. In this case, $T$ is called a \emph{support tree} of $N$ and $\tilde{T}$ a \emph{base tree} of $N$. 

\end{definition} 

As a support tree $T$ is a spanning tree of $N=(V,E)$, each $T$ is specified by its edge-set. This leads to the natural question: Which subsets $S$ of $E$ yield a support tree $N[S]=(V, S)$ of $N$?
Francis and Steel~\cite{FrancisSteel-2015-WhichPhylogeneticNetworks} proved that such ``admissible'' subsets  $S$ of $E$ are characterized by the following conditions (C1)--(C3).  As in \cite{Hayamizu-2021-StructureTheoremRooted}, we slightly generalize the original definition in \cite{FrancisSteel-2015-WhichPhylogeneticNetworks} so that we can consider admissible edge selections for any subgraph of $N$.  

\begin{definition}\label{def:admissible}
 Let $N$ be a rooted binary phylogenetic network and let $G$ be any subgraph of $N$. A subset $S$ of $E(G)$ is \emph{admissible} if it satisfies the following conditions:
  \begin{description}
    \item [(C1)] If $(u, v)$ is an edge of $G$ with $\outdeg_N(u) = 1$ or $\indeg_N(v) = 1$, then $S$ contains $(u,v)$.
    \item [(C2)] If $e_1$ and $e_2$ are distinct edges of $G$ with  $\tail(e_1) = \tail(e_2)$, then $S$ contains at least one of $\{e_1, e_2\}$.
    \item [(C3)] If $e_1$ and $e_2$ are distinct edges of $G$ with $\head(e_1) = \head(e_2)$, then $S$ contains exactly one of $\{e_1, e_2\}$.
  \end{description}
\end{definition}

\begin{theorem}[\cite{FrancisSteel-2015-WhichPhylogeneticNetworks}]\label{thm:bijection}
   Let $N$ be a rooted binary phylogenetic network and let $S \subseteq E(N)$. Then, the subgraph $N[S]$ of $N$ induced by $S$ is a support tree of $N$ if and only if $S$ is an admissible subset of $E(N)$.  
   Moreover, there exists a one-to-one correspondence between the family of admissible subsets $S$ of $E(N)$ and the family $\mathcal{T}$ 
   of support trees of $N$. 
   \end{theorem}

We note that two different admissible subsets $S_1, S_2 \subseteq E(N)$ can induce two isomorphic support trees $N[S_1]$ and $N[S_2]$ of $N$, and this raises subtly different counting problems on support trees (see Fig.\ 1 and Section 7.2 in \cite{Hayamizu-2021-StructureTheoremRooted}). Therefore, it may be more accurate to say that the family of admissible subsets $S$ of $E(N)$ corresponds to the family of edge-sets of support trees of $N$, rather than the set of such trees.  However, in this paper, based on the bijection described in Theorem \ref{thm:bijection}, we identify each support tree $T$ of $N$ with its edge-set $E(T)$  to discuss the number of admissible options for $E(T)$.

\subsection{The structure theorem for rooted binary phylogenetic networks}\label{sec:preliminaries1}
We recall the relevant materials from \cite{Hayamizu-2021-StructureTheoremRooted}. 
Given a rooted binary phylogenetic network $N$,  a connected subgraph $Z$ of  $N$ with $m:=|E(Z)|\geq 1$  is a  \emph{zig-zag trail} (\emph{in $N$}) if the edges of $Z$ can be permuted as $(e_1, \dots, e_m)$ such that either $\head(e_i)=\head(e_{i+1})$ or $\tail(e_i)=\tail(e_{i+1})$ holds for each $i\in [1,m-1]$. 
A zig-zag trail is represented by an alternating sequence of (not necessarily distinct) vertices and distinct edges, e.g., $(v_0, (v_0, v_1), v_1, (v_2, v_1), v_2, \dots, (v_m, v_{m-1}), v_m)$, which can more concisely written as $v_0 > v_1 < v_2 > \cdots > v_{m-1} < v_m$ or its reverse. 
	A zig-zag trail $Z$ in $N$ is  \emph{maximal} if $N$ contains no zig-zag trail $Z^\prime$  such that $Z$ is a proper subgraph of $Z^\prime$. Each maximal zig-zag trail in $N$ falls into one of the following four types. 
	A \emph{crown} is a maximal zig-zag trail $Z$ that has even $m:=|E(Z)|\geq 4$ and can be written in the cyclic form $v_0 < v_1 > v_2 < v_3 > \cdots  > v_{m-2} < v_{m-1} > v_{m} = v_0$.  A \emph{fence} is any maximal zig-zag trail that is not a crown. An \emph{N-fence} is a fence $Z$ with odd $m:=|E(Z)|\geq 1$,  expressed as $v_0 > v_1 < v_2 > \cdots < v_{m-1} > v_m$. A fence $Z$ with even $m:=|E(Z)|\geq 2$ is an \emph{M-fence} if expressed as $v_0 < v_1 > v_2 < \cdots < v_{m-1} > v_m$, and it is a \emph{W-fence} if expressed as $v_0 > v_1 < v_2 > \cdots > v_{m-1} < v_m$. 
	
		As in the approach taken in \cite{Hayamizu-2021-StructureTheoremRooted, HayamizuMakino-2023-RankingTopkTrees}, we often consider a maximal zig-zag trail $Z$ as a sequence $(e_1,\dots,e_{|E(Z)|})$ of edges, ordered according to their appearance in the trail. When $Z$ is a fence, the \emph{terminal edges} of $Z$ are the first and last edges in such a sequence. 
		This representation allows us to encode any edge-induced subgraph of $Z$ as a 0-1 sequence of length $|E(Z)|$. For example, for an $N$-fence $Z = (e_1, e_2, e_3, e_4, e_5)$ and a subset $\{e_1, e_3, e_5\} \subseteq E(Z)$, the subgraph of $Z$ induced by this edge-set can be written as $\langle1\; 0\; 1\;0\;1\rangle = \langle1(01)^2\rangle$.

We now restate the structure theorem for rooted binary phylogenetic networks, which provides a way to canonically decompose such networks $N$ (see Figure \ref{fig:zigzagtrail decomposition} for an illustraction) and characterizes the family of admissible subsets $S$ of $E(N)$, namely, the family of edge-sets of support trees of $N$. 

\begin{theorem}[\cite{Hayamizu-2021-StructureTheoremRooted}]\label{thm:structure}
For any rooted binary phylogenetic network $N$, the following hold: 
\begin{enumerate}
	\item The maximal zig-zag trail decomposition $\mathcal{Z}=\{Z_1,\dots, Z_d\}$  of  $N$ is unique to $N$. 
	\item Given $S \subseteq E(N)$, $N[S]$ is a support tree of $N$ if and only if for each $Z_i\in \mathcal{Z}$, $E(G)\cap E(Z_i)$ is an admissible subset of $E(Z_i)$. Here, the family $\mathcal S(Z_i)$ of admissible subsets of $E(Z_i)$ is given by \eqref{sequence}. 
	\item The family $\mathcal T$ of support trees of $N$ is non-empty (i.e., $N$ is tree-based) if and only if each $\mathcal{S}(Z_i)$ is non-empty (i.e., no element of $\mathcal{Z}$ is a W-fence). When $\mathcal T$ is non-empty, it is characterized by $\mathcal T=\prod_{i=1}^{d}{\mathcal S(Z_i)}$, where $(Z_1,\dots, Z_d)$ is an arbitrary ordering for $\mathcal{Z}$. 
\end{enumerate}
\begin{align}
  \label{sequence}
    \mathcal S (Z_i) = \begin{cases}
    	\emptyset & \text{if $Z_i$ is a W-fence}\\
        \{\langle 1(01)^{(|E(Z_i)|-1)/2}\rangle  \} & \text{if $Z_i$ is an N-fence}\\
        \{\langle (10)^{|E(Z_i)|/2}\rangle, \langle (01)^{|E(Z_i)|/2}\rangle \} & \text{if $Z_i$ is a crown}\\
        \{\langle 1(01)^p (10)^q 1\rangle \mid p, q\in \mathbb{Z}_{\geq 0}, p,+q = (|E(Z_i)|-2)/2 \} & \text{if $Z_i$ is an M-fence}
      \end{cases}
\end{align}
\end{theorem}

\begin{proposition}[\cite{Hayamizu-2021-StructureTheoremRooted}]\label{prop:linear.zzt}
The maximal zig-zag trail decomposition $\mathcal{Z}$ of $N$ can be computed in $\Theta(|E(N)|)$ time. 	
\end{proposition}

\begin{figure}[htbp]
    \centering
    \includegraphics[width=0.6\textwidth]{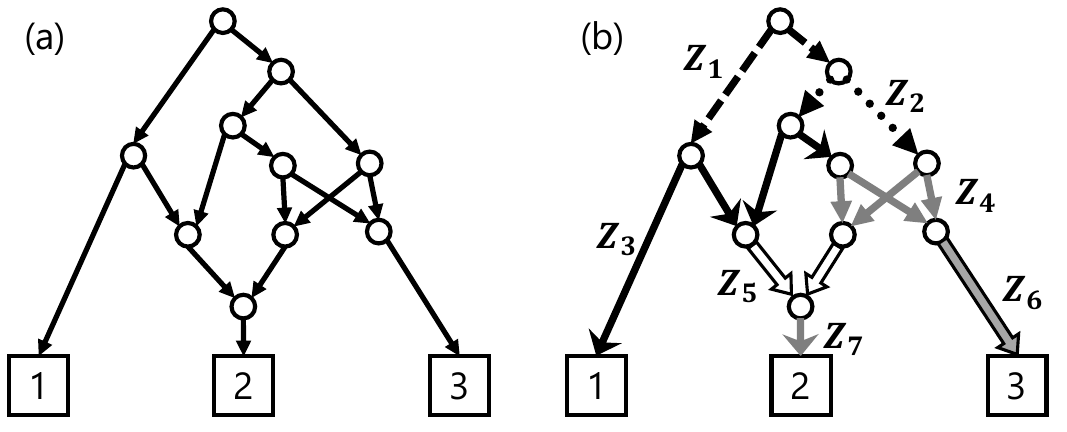}
    \caption{(a) A rooted binary phylogenetic network $N$. (b) The maximal zig-zag trail decomposition $\mathcal Z = \{Z_1, \dots, Z_7\}$  of $N$, where $Z_1$, $Z_2$ and $Z_3$ are M-fences, $Z_4$ is a crown, $Z_5$ is a W-fence, and $Z_6$ and $Z_7$ are N-fences.}
    \label{fig:zigzagtrail decomposition}
\end{figure}

	



Theorem~\ref{thm:structure} and Proposition \ref{prop:linear.zzt} have furnished optimal algorithms for various computational problems on support trees \cite{Hayamizu-2021-StructureTheoremRooted, HayamizuMakino-2023-RankingTopkTrees}.
For example, the number $|\mathcal{T}|$ of support trees of $N$ is given by $|\mathcal{T}|=\prod_{i=1}^d |\mathcal S(Z_i)|$, and using \eqref{eq:number.admissible.sets}, it can be computed in $\Theta(|E(N)|)$ time.
\begin{align}
  \label{eq:number.admissible.sets}
      |\mathcal S(Z_i)| = \begin{cases}
      0 & \text{$Z_i$ is an W-fence}\\
        1 & \text{if $Z_i$ is an N-fence}\\
        2 & \text{if $Z_i$ is a crown}\\
        |E(Z_i)|/2 & \text{if $Z_i$ is an M-fence}\\
        \end{cases}
    \end{align}

\section{The family of support networks and its two subfamilies}
\label{sec:Problem}
For a rooted (not-necessarily-binary) phylogenetic network $N=(V, E)$, 
$r(N)$ denotes the \emph{reticulation number} of $N$, i.e., the number of reticulations in $V$, and $\mathrm{level}(N)$ denotes the \emph{level} of $N$, i.e., the maximum number of reticulations of $N$ contained in a block of $N$, and the \emph{tier} is $|E|-|V|+1$. By definition, $\mathrm{level}(N) \leq r(N)$ holds. 
In general,  $r(N)\leq |E|-|V|+1$ holds. 
If $N$ is binary, or if $N$ is a non-binary network where each $v \in V$ satisfies $\indeg_N(v)\leq 2$ and $\outdeg_N(v)\leq 2$, then the hand-shaking lemma gives $r(N)=|E|-|V|+1$. Since we restrict our attention to binary networks in this paper, we use the terms tier and reticulation number interchangeably.

Let $N$ and $\tilde{G}$ be rooted binary phylogenetic networks on $X$ and let $G$ be a spanning subgraph  of $N$. If $\tilde{G}$ is obtained by smoothing all $v\in V(G)$ with $\indeg_G(v)=\outdeg_G(v)=1$, then $G$ is a \emph{support network} of $N$, and $\tilde{G}$ is a \emph{base network} of $N$. 
Unlike support and base trees, support and base networks always exist since $N$ itself is a support network of $N$.
The \emph{base tier} of $N$, denoted by $r^\ast (N) $, is the minimum value of $r(G):=|E(G)|-|V(G)|+1$,  which equals $r(\tilde{G})$,  over all support networks $G$ of $N$. Similarly, the \emph{base level} of $N$, denoted by $\mathrm{level}^\ast (N) $, is the minimum value of $\mathrm{level}(G)$, which equals $\mathrm{level}(\tilde{G})$, over all support networks $G$ of $N$.   
If $k$ is the base tier (resp.\ base level) of $N$, then $N$ is  \emph{tier-$k$-based} (resp.\ \emph{level-$k$-based}). 


For unrooted phylogenetic networks, Fischer and Francis \cite{FischerFrancis-2020-HowTreebasedMy} has introduced the concept of support networks to discuss tier-$k$-based and level-$k$-based networks. While the definitions are analogous, the mathematical properties of these concepts differ between rooted and unrooted networks. Indeed, for unrooted $N$, it was shown in \cite{FischerFrancis-2020-HowTreebasedMy} that if $N$ has at most one non-trivial block, then $r^\ast (N) = \mathrm{level}^\ast(N)$ holds. For rooted $N$, this equality does not hold in general (see Figure \ref{fig:2 reticulation, level-1} for a counterexample).

\begin{figure}[htbp]
  \centering
  \includegraphics[width=.8\textwidth]{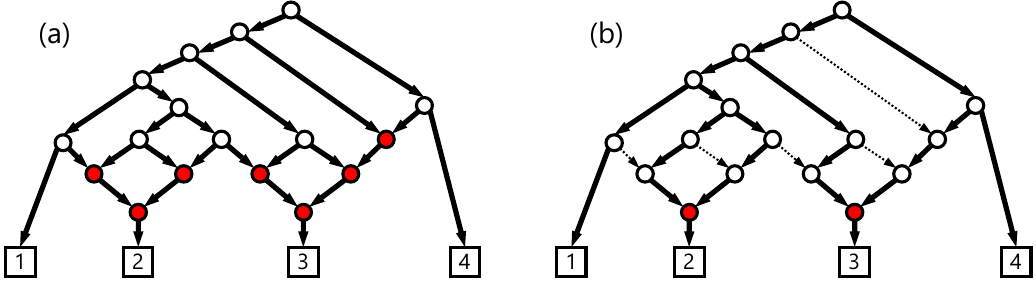}
  \caption{(a) A rooted binary phylogenetic network $N$ with $r^\ast(N)=2$ and $\mathrm{level}^\ast(N)=1$.  
           (b) A support network $G$ of $N$ that attains the optimal values $r(G) = 2$ and $\mathrm{level}(G) = 1$. The edges of $G$ are shown by solid arrows. The reticulations in each graph are colored red.}
  \label{fig:2 reticulation, level-1}
\end{figure}

The focus of this paper is on the computation of  $r^\ast (N)$ and $\mathrm{level}^\ast(N)$ for rooted $N$, with particular emphasis on the latter which presents greater computational challenges. We approach these problems using the families of minimal and minimum support networks of $N$, which are defined as follows:

\begin{definition}
\label{def:ABC}
For a rooted binary phylogenetic network $N$,  
 \begin{itemize}
 	\item $\mathcal{A}_N$ is the set of all support networks of $N$; 
 	\item $\mathcal{B}_N$ is the set of minimal support networks of $N$ (i.e. those with a minimal edge-set); 
 	\item $\mathcal{C}_N$ is  the set of minimum support networks of $N$ (i.e. those with the fewest edges).
 \end{itemize}
\end{definition}

Definition \ref{def:ABC} implies $(\emptyset \neq ) \mathcal{A}_N\supseteq \mathcal{B}_N\supseteq \mathcal{C}_N$. 
As Figure \ref{fig:example} shows, these families can be all distinct.  
We also note that if $N$ is tree-based, then $\mathcal{C}_N$ is nothing but the set $\mathcal{T}$ of support trees of $N$ characterized in Theorems \ref{thm:bijection} and \ref{thm:structure}. 

\begin{figure}[htbp]
  \centering
  \includegraphics[width=0.95\textwidth]{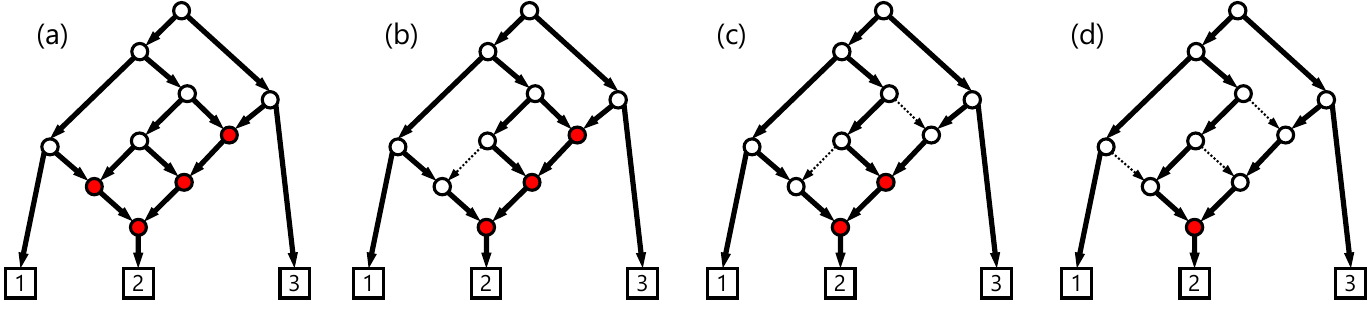}
  \caption{
  (a): A rooted binary phylogenetic network $N$. (b): A support network in $\mathcal{A}_N \setminus  \mathcal{B}_N$. (c): A support network in $\mathcal{B}_N \setminus \mathcal{C}_N$.  (d): A support network in $\mathcal{C}_N$. The reticulations in each graph are colored red.}
  \label{fig:example}
\end{figure}

\section{Counting the support networks of three types}\label{sec:count.main}
Examining the proper subfamilies $\mathcal{B}_N$ and $\mathcal{C}_N$ rather than the entire family $\mathcal{A}_N$ of support networks could reduce the search space for computing $r^\ast(N)$ or $\mathrm{level}^\ast (N)$. A natural question is how substantial this reduction might be. In this section, we generalize Theorem~\ref{thm:structure} to derive analogous direct-product characterizations of $\mathcal{A}_N$, $\mathcal{B}_N$, and $\mathcal{C}_N$, and determine their cardinalities. To achieve this, we relax condition (C3) in Definition~\ref{def:admissible} to introduce different notions of admissibility as follows.

\begin{definition}
  \label{def:alpha-admissible}
  Let $N$ be a rooted binary phylogenetic network and let $G$ be any subgraph of $N$. A subset $S$ of $E(G)$ is \emph{$\mathcal A$-admissible} if it satisfies the following conditions:
  \begin{description}
    \item [(C1$^\star$)] If $(u, v)$ is an edge of $G$ with $\outdeg_N(u) = 1$ or $\indeg_N(v) = 1$, then $S$ contains $(u,v)$; 
    \item [(C2$^\star$)] If $e_1$ and $e_2$ are distinct edges of $G$ with  $\tail(e_1) = \tail(e_2)$ or $\head(e_1) = \head(e_2)$, then $S$ contains at least one of $\{e_1, e_2\}$.
  \end{description}
 In particular, given an $\mathcal A$-admissible subset $S$ of $E(G)$,  $S$ is \emph{$\mathcal B$-admissible} if it is minimal, i.e. no proper subset of $S$ is an $\mathcal A$-admissible subset of $E(G)$, and  $S$ is \emph{$\mathcal C$-admissible} if it is the smallest among all $\mathcal A$-admissible subsets of $E(G)$.
\end{definition}

 By the maximality of each zig-zag trail $Z_i$ in $N$, we have $\outdeg_N(u) = \outdeg_{Z_i}(u) = 1$ and  $\indeg_N(v) = \indeg_{Z_i}(v) = 1$. 
Then, $S \subseteq E(Z_i)$ satisfies (C1$^\star$) in Definition~\ref{def:alpha-admissible} if and only if $S$ contains both terminal edges of each fence $Z_i \in\mathcal{Z}$. This is the same requirement as in (C1) in Definition~\ref{def:admissible}. By contrast, unlike (C3) in Definition~\ref{def:admissible}, (C2$^\star$) only requires that for any two consecutive edges of $Z_i$ with a common head or tail, $S$ contains at least one of them. In other words,  (C2$^\star$)  allows edge selections that yield reticulations, and for this difference, 
an $\mathcal{A}$-admissible subset of $E(N)$ can induce non-tree graphs. 

Lemma \ref{thm:bijection.ABC} is a generalization of Theorem \ref{thm:bijection} and of the second statement of Theorem \ref{thm:structure}. We omit the proof here.

\begin{lemma}
  \label{thm:bijection.ABC}
  Let $N$ be a rooted binary phylogenetic network, and let $\mathcal{X} \in \{\mathcal{A}, \mathcal{B}, \mathcal{C}\}$. 
  Then, there is a one-to-one correspondence between the family $\mathcal{X}_N$ of support networks and the family of $\mathcal{X}$-admissible subsets of $E(N)$. Moreover,  the subgraph $N[S]$ of $N$ induced by $S\subseteq E(N)$ is a support network in $\mathcal{X}_N$ if and only if $S \cap E(Z_i)$ is an $\mathcal{X}$-admissible subset of $E(Z_i)$ for each maximal zig-zag trail $Z_i$ in $N$.
\end{lemma}

%

\subsection{The number $|\mathcal{A}_N|$ of all support networks}
\label{sec: all support network}
 As explained above, one can construct an $\mathcal{A}$-admissible subset of $E(N)$ simply by obeying (C1$^\star$) and (C2$^\star$) in Definition \ref{def:alpha-admissible}, namely, by selecting both terminal edges of each fence and by avoiding two consecutive unselected edges in each $Z_i$. Thus, the family $\mathcal{S}_{\mathcal{A}}(Z_i)$ of $\mathcal{A}$-admissible subsets of $E(Z_i)$ is expressed by \eqref{eq:sequence for support network} for each $Z_i\in \mathcal{Z}$.  
    \begin{align}
      \mathcal{S}_{\mathcal{A}}{(Z_i)} = \begin{cases}
        \{ \langle b_1\cdots b_{|E(Z_i)|}\rangle \mid \text{$b_1 = b_{|E(Z_i)|} = 1$, $\langle 00\rangle$ is not a subsequence}\}\\
        \hfill \text{if $Z_i$ is a fence}\\
        \{  \langle b_1 \cdots b_{|E(Z_i)|}\rangle \mid \text{$\langle 00\rangle$ is not a subsequence for any ordering} \}\\
        \hfill \text{if $Z_i$ is a crown}
        \end{cases}\label{eq:sequence for support network}
    \end{align}
    
    By Lemma \ref{thm:bijection.ABC},  $\mathcal{A}_N$ is characterized by $\mathcal{A}_N = \prod_{i=1}^d \mathcal{S}_{\mathcal{A}} (Z_i)$. Therefore, $|\mathcal{A}_N| = \prod_{i=1}^d |\mathcal{S}_{\mathcal{A}} (Z_i)|$. As in \eqref{eq:sequence for support network}, each number $|\mathcal{S}_{\mathcal{A}} (Z_i)|$ depends on whether $Z_i$ is a fence or not. This  number can be more explicitly represented using Fibonacci and Lucas numbers (OEIS A000045 and OEIS A000032, respectively) as follows. The proof of Theorem \ref{thm:fibonacci.lucas} is given in Appendix \ref{apd.proof:fibonacci.lucas}.

  
  \begin{theorem}\label{thm:fibonacci.lucas}
    Let $N$ be a rooted binary phylogenetic network and $\mathcal{Z} = \{Z_1, \dots, Z_d\}$ be the maximal zig-zag trail decomposition of $N$. 
    Let $\{F_n\}$ be the Fibonacci sequence defined by $F_1 = F_2 = 1$ and $F_n = F_{n-1} + F_{n-2}$ ($n \geq 3$), and $\{L_n\}$ be the Lucas sequence  defined by $L_1 = 1, L_2 = 3$ and $L_n = L_{n-1} + L_{n-2}$ ($n\geq 3$).
    Then, the number of support networks $|\mathcal{A}_N|$ is given by \eqref{eq:count all support}. 
    \begin{align}
      \label{eq:count all support}
      |\mathcal{A}_N| = \prod_{Z_i: \text{fence}} F_{|E(Z_i)|} \cdot \prod_{Z_i: \text{crown}} L_{|E(Z_i)|}
    \end{align}
    Moreover, $|\mathcal{A}_N| = \Theta(\phi^{|E(N)|})$ holds, where $\phi = (1+\sqrt{5})/2 = 1.6180\dots$ is the golden ratio.
  \end{theorem}


  
  Theorem \ref{thm:fibonacci.lucas} yields an obvious algorithm for computing $|\mathcal{A}_N|$. It first computes the maximal zig-zag trail decomposition $\mathcal{Z} = \{Z_1, \dots, Z_d\}$ of $N$, and then calculates $|\mathcal{A}_N|$ using \eqref{eq:count all support}. 
  
  \begin{proposition}
  	    \label{prop:complexity.count.A}
     $|\mathcal{A}_N|$ can be computed by in $\Theta(|E(N)|)$ time. 
  \end{proposition}
  \begin{proof}
  By Proposition \ref{prop:linear.zzt}, $\mathcal{Z}$ can be computed in $\Theta(|E(N)|)$ time. 
   For each $Z_i\in \mathcal{Z}$, deciding whether $Z_i$ is a crown or not takes $O(|E(Z_i)|)$ time, and $F_{|E(Z_i)|}$ and $L_{|E(Z_i)|}$ can be computed in $O(|E(Z_i)|)$ time. Since $|\mathcal{Z}|=O(|E(N)|)$, one can calculate $|\mathcal{A}_N|$ using \eqref{eq:count all support}  in $O(|E(N)|)$ time.
    Loading $N$ takes $\Omega(|E(N)|)$ time, so the overall complexity is $\Theta(|E(N)|)$. 
  \end{proof}
  
  Similarly, we can develop an algorithm to generate all (or a desired number of) elements of $\mathcal{A}_N$ by sequentially outputting each element of  $\prod_{i=1}^d \mathcal{S}_{\mathcal{A}} (Z_i)$, achieving $\Theta(|E(N)|)$ delay. We refer the reader to  \cite{Hayamizu-2021-StructureTheoremRooted} or \cite{HayamizuMakino-2023-RankingTopkTrees} for the basics on the complexity analysis of listing algorithms.

  

\subsection{The number $|\mathcal{B}_N|$ of minimal support networks}
\label{sec:minimal support network}
We can see that $\mathcal{A}$-admissible subset $S$ of $E(Z_i)$ is $\mathcal{B}$-admissible if and only if $S$ contains no three consecutive edges in the edge sequence $(e_1,\dots, e_{|E(Z_i)|})$ of $Z_i$. 
Thus, the family $\mathcal{S}_{\mathcal{B}}(Z_i)$ of $\mathcal B$-admissible subsets of $E(Z_i)$ is expressed as in \eqref{eq:sequence for minimal support network} for each $Z_i\in \mathcal{Z}$.
\begin{align}
  \mathcal{S}_{\mathcal{B}}(Z_i) =\begin{cases}
    \{ \langle b_1\cdots b_{|E(Z_i)|}\rangle \mid \text{$b_1 = b_{|E(Z_i)|} = 1$, neither $\langle 00\rangle$ nor $\langle 111\rangle$ is a subsequence}\}\\
    \hfill \text{if $Z_i$ is a fence}\\
    \{ \langle b_1 \cdots b_{|E(Z_i)|}\rangle \mid \text{neither $\langle 00\rangle$ nor $\langle 111\rangle$ is a subsequence for any ordering} \}\\
    \hfill \text{if $Z_i$ is a crown}
    \end{cases} \label{eq:sequence for minimal support network}
\end{align}

By Lemma \ref{thm:bijection.ABC}, $\mathcal{B}_N$ is characterized by $\mathcal{B}_N = \prod_{i=1}^d \mathcal{S}_{\mathcal{B}} (Z_i)$.
Therefore, $|\mathcal{B}_N| = \prod_{i=1}^d |\mathcal{S}_{\mathcal{B}} (Z_i)|$. This number is expressed using Padovan numbers (OEIS A000931) and Perrin numbers (OEIS A001608) as follows.
The proof of Theorem \ref{thm:counting by padovan, perrin} is in Appendix \ref{apd.proof.B.padovan.perrin}.

  \begin{theorem}\label{thm:counting by padovan, perrin}
    Let $N$ be a rooted binary phylogenetic network and $\mathcal{Z} = \{Z_1, \dots, Z_d\}$ be the maximal zig-zag trail decomposition of $N$. 
    Let $\{P_n\}$ be the Padovan sequence defined by $(P_1, P_2, P_3) = (1, 1, 1)$ and $P_n = P_{n-2} + P_{n-3}$ ($n \geq 4$), and $\{Q_n\}$ be the Perrin sequence defined by $(Q_1, Q_2, Q_3) = (0, 2, 3)$ and $Q_n = Q_{n-2} + Q_{n-3}$ ($n \geq 4$).
    Then, the number of minimal support networks $|\mathcal{B}_N|$ is given by \eqref{eq:count minimal support}. 
    \begin{align}
      \label{eq:count minimal support}
      |\mathcal{B}_N| = \prod_{Z_i: \text{fence}} P_{|E(Z_i)|} \cdot \prod_{Z_i: \text{crown}} Q_{|E(Z_i)|}
    \end{align}
    Moreover, $|\mathcal{B}_N| = \Theta(\psi^{|E(N)|})$ holds, where $\psi = \sqrt[3]{(9 +\sqrt{69})/18}+ \sqrt[3]{(9 -\sqrt{69})/18} =  1.3247\dots$ is the plastic number.
    \end{theorem}

    \begin{proposition}
    \label{prop:complexity of count minimal support network}
   $|\mathcal{B}_N|$ can be computed in $\Theta(|E(N)|)$ time. 
  \end{proposition}
  \begin{proof}
   The only difference between \eqref{eq:count all support} and \eqref{eq:count minimal support} is that $F_{|E(Z_i)|}$ and $L_{|E(Z_i)|}$ are replaced by $P_{|E(Z_i)|}$ and $Q_{|E(Z_i)|}$, respectively. For each $Z_i\in \mathcal{Z}$, both $P_{|E(Z_i)|}$ and $Q_{|E(Z_i)|}$ can also be computed in $O(|E(Z_i)|)$ time. The remainder is the same as the proof of Proposition \ref{prop:complexity.count.A}.
  \end{proof}
  
As noted in Section \ref{sec: all support network}, the elements of $\mathcal{B}_N$ can also be listed with $\Theta(|E(N)|)$ delay.

\subsection{The number $|\mathcal{C}_N|$ of minimum support networks}
\label{sec:minimum support network}
The construction of $\mathcal{C}$-admissible subsets of each $E(Z_i)$ is the same as that of admissible subsets for support trees, with the only difference that $\mathcal{C}$-admissible subsets allow W-fences. 
\begin{align}\label{eq:sequence for minimum support networks}
  \mathcal{S}_{\mathcal{C}} (Z_i) = \begin{cases}
    \{\langle (10)^{|E(Z_i)|/2}\rangle, \langle (01)^{|E(Z_i)|/2}\rangle \} \hfill \text{if $Z_i$ is a crown}\\
    \{\langle 1(01)^{(|E(Z_i)|-1)/2}\rangle  \} \hfill \text{if $Z_i$ is an N-fence}\\
    \{\langle 1(01)^p (10)^q 1\rangle \mid p, q\in \mathbb{Z}_{\geq 0}, p+q = (|E(Z_i)|-2)/2 \}\qquad\qquad\qquad\\
    \hfill \text{if $Z_i$ is an M-fence or W-fence}
  \end{cases}
\end{align}

    By Lemma \ref{thm:bijection.ABC},  $\mathcal{C}_N$ is characterized by $\mathcal{C}_N = \prod_{i=1}^d \mathcal{S}_{\mathcal{C}} (Z_i)$. Therefore, $|\mathcal{C}_N| = \prod_{i=1}^d |\mathcal{S}_{\mathcal{C}} (Z_i)|$. 
From \eqref{eq:sequence for minimum support networks}, we  have $|\mathcal{S}_{\mathcal{C}}(Z_i)| = 2$ for any crown $Z_i$, $|\mathcal{S}_{\mathcal{C}}(Z_i)| = 1$ for any N-fence  $Z_i$, and $|\mathcal{S}_{\mathcal{C}}(Z_i)| = |E(Z_i)|/2$ for any other $Z_i$. Thus, we obtain Theorem \ref{cor:count minimum support network} and Proposition \ref{prop:count.C.linear}. 
\begin{theorem}
  \label{cor:count minimum support network}
  Let $N$ be a rooted binary phylogenetic network and $\mathcal{Z} = \{Z_1, \dots, Z_d\}$ be the maximal zig-zag trail decomposition of $N$. Then,  $|\mathcal{C}_N| = 2^c\cdot \prod_{Z_i\in \mathcal{Z}_{MW}} (|E(Z_i)|/2) $ holds, where $c$ is the number of crowns in $\mathcal{Z}$, and $\mathcal{Z}_{MW}:=\{Z_i\in \mathcal{Z} \mid Z_i \text{ is an M-fence or W-fence}\}$. 
\end{theorem}

\begin{proposition}\label{prop:count.C.linear}
    $|\mathcal{C}_N|$ can be computed in $\Theta(|E(N)|)$ time.
\end{proposition}

As noted in Section \ref{sec: all support network},  the elements of $\mathcal{C}_N$ can also be listed with $\Theta(|E(N)|)$ delay.

\subsection{Counting support networks in random phylogenetic networks}\label{sec:experiment.case.study}
To examine the typical behavior of support network counts in practice, 
we generated rooted binary phylogenetic networks with $n$ leaves and $r = 2(n - 1)$ reticulations, which simulates a moderately reticulated evolution. 
For each value of $n$ ranging from $3$ to $10$, we generated $100$ random networks that may or may not be tree-based. For each sampled network $N$, we computed the number of all support networks $|\mathcal{A}_N|$, the number of minimal support networks $|\mathcal{B}_N|$, and the number of minimum support networks $|\mathcal{C}_N|$ using our linear-time algorithms proposed in Sections \ref{sec: all support network}, \ref{sec:minimal support network}, and \ref{sec:minimum support network}. The minimum, maximum, and median values of each counts across the $100$ samples  are summarized in Table~\ref{tab:experiment 1}. The code for generating sample networks, the set of generated samples, and the counts for each sample are available in our GitHub repository.

Although these three quantities grow exponentially in the size of $N$, Table~\ref{tab:experiment 1} shows that
  $|\mathcal{B}_N|$  is consistently much smaller than  $|\mathcal{A}_N|$, and $|\mathcal{C}_N|$ is even smaller. Notably, the number of support networks varies widely even among networks with the same $n$ and $r$. In other words, some networks admit only a few, while others admit very many support networks. This highlights the diversity of phylogenetic networks with the same size.

\begin{table}[ht]
  \centering
  \begin{tabular}{crrrrrrrrr}
  \hline
  \multirow{2}{*}{$n$} & \multicolumn{3}{c}{$|\mathcal{A}_N|$} & \multicolumn{3}{c}{$|\mathcal{B}_N|$} & \multicolumn{3}{c}{$|\mathcal{C}_N|$} \\ \cline{2-10}
                     & Min & Max & Median & Min & Max & Median & Min & Max & Median \\ \hline
  3                  & 15  & 64  & 30     & 2   & 10  & 4      & 1   & 9   & 3      \\ \hline
  4                  & 55  & 336 & 160    & 4   & 28  & 9      & 1   & 18  & 4      \\ \hline
  5                  & 288 & 2,640& 900    & 8   & 72  & 24     & 1   & 36  & 8      \\ \hline
  6                  & 900& 15,840& 4,478   & 6   & 256 & 48     & 1   & 108 & 12     \\ \hline
  7                  & 6,435& 134,784& 24,720  & 24 & 768 & 142     & 2   & 192 & 24     \\ \hline
  8                  & 34,272& 571,914& 134,400 & 40  & 1,680& 270    & 2   & 576 & 38     \\ \hline
  9                  & 93,600& 4,186,080& 699,920 & 90 & 2,880& 576      & 2   & 1,152& 68     \\ \hline
  10                 & 449,280& 21,715,200& 3,861,936& 144 & 7,056& 1,372   & 4   & 1,728& 124    \\ \hline
  \end{tabular}
  \caption{The numbers of support networks of three types  for 100 randomly generated rooted binary phylogenetic networks $N$ with $n$ leaves and $r = 2(n-1)$ reticulations.}
  \label{tab:experiment 1}
  \end{table}

\section{Finding a support network with the fewest reticulations}\label{sec:minimum.reticulation}
As defined in Section \ref{sec:Problem}, the base tier  $r^\ast (N)$ of $N$ is the minimum value of $r(G)$ over all support networks $G$ of $N$. This leads us to Problem \ref{prob:minimum reticulation}.

\begin{problem}[\textsc{Reticulation Minimization}]
  \label{prob:minimum reticulation}
  Given a rooted binary phylogenetic network $N$, compute the base tier $r^\ast(N)$, and find a support network $G$ of $N$ with $r(G)=r^\ast(N)$.
\end{problem}

Since $r(G):=|E(G)| - |V(G)| + 1$ and  $G$ is a spanning subgraph of $N$, the support networks $G$ of $N$ minimize $r(G)$ if and only if it has the fewest edges among all support networks, i.e., it is a minimum support network. Therefore, an optimal solution of Problem \ref{prob:minimum reticulation} can be obtained by simply selecting an arbitrary element from $\mathcal{C}_N = \prod_{i=1}^d \mathcal{S}_{\mathcal{C}}(Z_i)$. This is done by computing the maximal zig-zag trail decomposition $\{Z_1, \dots, Z_d\}$ of $N$, and then selecting an arbitrary element of $\mathcal{S}_{\mathcal{C}}(Z_i)$ expressed in  \eqref{eq:sequence for minimum support networks} for each $Z_i$.

For example, the network $N$ in Figure~\ref{fig:zigzagtrail decomposition} is decomposed into the maximal zig-zag trails $Z_1$ through $Z_7$. According to \eqref{eq:sequence for minimum support networks}, the sequence $\langle 11 \rangle$ is $\mathcal{C}$-admissible for $E(Z_1)$, $E(Z_2)$, and $E(Z_5)$; $\langle 1010 \rangle$ is $\mathcal{C}$-admissible for $E(Z_3)$ and $E(Z_4)$; and $\langle 1 \rangle$ is $\mathcal{C}$-admissible for $E(Z_6)$ and $E(Z_7)$. Piecing these $\mathcal{C}$-admissible subsets together yields the edge-set of a support network $G^\ast$ of $N$ with the minimum reticulation number, that is, $r(G^\ast) = 1$.

\begin{theorem}
  \label{thm:linear-time.ret.min}
  The above algorithm solves Problem~\ref{prob:minimum reticulation} in $\Theta(|E(N)|)$ time.
\end{theorem}

When $N$ is tree-based, $r^\ast(N)=0$ holds and a support network $G$ with $r(G^\ast)=0$ is simply its support tree. In general, $r(G^\ast)$ equals the number of W-fences in the maximal trail decomposition of $N$.

\section{Finding a support network with the minimum level}\label{sec:minimum.level}
In Section \ref{sec:Problem}, we have also defined $\mathrm{level}^\ast(N)$, i.e., the base level of $N$.    
In contrast to Problem~\ref{prob:minimum reticulation}, we conjecture that Problem~\ref{prob:level opt} is NP-hard. Here, we provide exact and heuristic exponential algorithms for solving Problem~\ref{prob:level opt}. 

\begin{problem}[\textsc{Level Minimization}]
  \label{prob:level opt}
    Given a rooted binary phylogenetic network $N$, compute the base level, $\mathrm{level}^\ast(N)$, and find a support network $G$ with $\mathrm{level}(G)=\mathrm{level}^\ast(N)$.
 \end{problem}

\subsection{Exact algorithm}
An obvious exact algorithm for Problem \ref{prob:level opt}  computes $\mathrm{level}(G)$ of each element $G$ in $\mathcal{A}_N$ to obtain the minimum value, but this method is clearly impractical. Algorithm \ref{alg:ExactAlgorithm} performs an exhaustive search in $\mathcal{B}_N$ instead of $\mathcal{A}_N$. Recalling the numerical results in Table \ref{tab:experiment 1}, we know that this search space reduction is significant. We can prove that $\mathcal{B}_N$ contains a correct $G^\ast$,  yielding  
 Theorem \ref{thm:alg1}. The proof  is given in Appendix.

\begin{algorithm}
  \caption{Exact method for solving \textsc{Level Minimization} (Problem \ref{prob:level opt})}
  \label{alg:ExactAlgorithm}
  \KwIn{A rooted binary phylogenetic network $N$}
  \KwOut{A support network $G^\ast$ of $N$ with the minimum level, the value of $\mathrm{level}^\ast(N)$}
  \ForEach{$G \in \mathcal{B}_N$}{
    compute $\mathrm{level}(G)$ 
  }
  let $G^\ast$ be a support network $G$ with the minimum $\mathrm{level}(G)$ in $\mathcal{B}_N$\;
  \Return $G^\ast$, and $\mathrm{level}(G^\ast)$ as $\mathrm{level}^\ast(N)$
\end{algorithm}

\begin{theorem}\label{thm:alg1}
Algorithm \ref{alg:ExactAlgorithm} returns a correct solution of Problem \ref{prob:level opt}. It runs  in $O(|\mathcal{B}_N|\cdot |E(N)|) = O(\psi^{|E(N)|}\cdot |E(N)|)$ time, where $\psi$ is as in Theorem~\ref{thm:counting by padovan, perrin}. 
\end{theorem}

\subsection{Heuristic algorithm}

%

Algorithm~\ref{alg:HeuristicAlgorithm} performs an exhaustive search in an even smaller search space $\mathcal{C}_N$. It repeats the same procedure as Algorithm~\ref{alg:ExactAlgorithm}, except its iteration number is reduced from $|\mathcal{B}_N|$ to  $|\mathcal{C}_N|$. 
  By Theorem \ref{thm:alg1}, Algorithm~\ref{alg:HeuristicAlgorithm} therefore runs in $O(|\mathcal{C}_N|\cdot |E(N)|)$ time.
  
  Algorithm~\ref{alg:HeuristicAlgorithm} does not alway output a correct solution of Problem \ref{prob:level opt}, because minimizing the number of reticulations does not imply minimizing the level of support networks. Consider a level-$1$-based network $N$ in Figure \ref{fig:heuristic failure}(a), which has the property that any $G \in \mathcal{C}_N$ satisfies $r(G)=2$ and $\mathrm{level}(G)=2$. Then, Algorithm \ref{alg:HeuristicAlgorithm}  outputs a level-$2$ support network 
  as illustrated in Figure \ref{fig:heuristic failure}(b). An optimal support tree $G^\ast$ with $\mathrm{level}(G^\ast)=1$ exists in $\mathcal{B}_N\setminus \mathcal{C}_N$ while it is not optimal in terms of reticulation numbers, as shown in Figure \ref{fig:heuristic failure}(c). 

\begin{algorithm}
  \caption{Heuristic method for solving \textsc{Level Minimization} (Problem \ref{prob:level opt})}
  \label{alg:HeuristicAlgorithm}
  \KwIn{A rooted binary phylogenetic network $N$}
  \KwOut{A support network $G^\ast$ of $N$ with the minimum level, the value of $\mathrm{level}^\ast (N)$}
  \ForEach{$G \in \mathcal{C}_N$}{
    compute $\mathrm{level}(G)$ 
  }
  let $G^\ast$ be a support network $G$ with the minimum $\mathrm{level}(G)$ in $\mathcal{C}_N$\;
  \Return $G^\ast$, and $\mathrm{level}(G^\ast)$ as $\mathrm{level}^\ast(N)$
\end{algorithm}

\begin{figure}[htbp]
  \centering
  \includegraphics[width=0.8\textwidth]{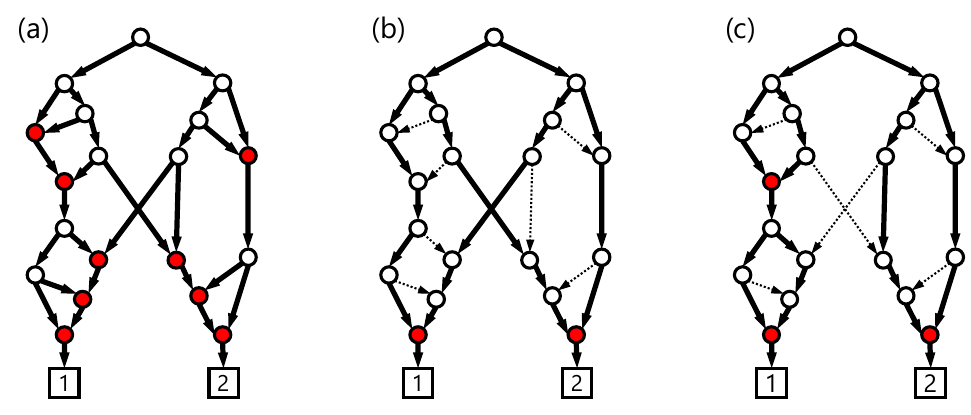}  
  \caption{(a) An instance of Problem \ref{prob:level opt} for which Algorithm \ref{alg:HeuristicAlgorithm} cannot find a correct solution. (b) An example of a support network output by Algorithm \ref{alg:HeuristicAlgorithm} with   
  two reticulations (red) and level $2$. 
  (c) An optimal support network 
  that has three reticulations (red) yet achieves level $1$.  
  \label{fig:heuristic failure}}
\end{figure}

\subsection{Performance evaluation of the exact and heuristic algorithms}

To evaluate the performance of the proposed algorithms, we conducted experiments on randomly generated rooted binary phylogenetic networks with $n = 8$ leaves and reticulation numbers $r$ ranging from $1$ to $43$. For each $r$, we generated 100 random networks. On each instance, we recorded the execution times of Algorithms~\ref{alg:ExactAlgorithm} and~\ref{alg:HeuristicAlgorithm}, and checked whether the two algorithms produced the same value of $\mathrm{level}^\ast(N)$. All experiments were performed on a Windows laptop equipped with an Intel Core i9-12900H (2.50 GHz) CPU and 32~GB of memory. A summary of the results is shown in Figure~\ref{fig:experiment results}.

The results confirm that the runtime of both algorithms increases exponentially with the reticulation number $r$. Algorithm~\ref{alg:ExactAlgorithm} was executed up to $r = 34$, but was terminated beyond this point due to insufficient memory.
Algorithm~\ref{alg:HeuristicAlgorithm} was executed up to $r = 43$ for the same reason, with its runtime consistently between $1/10$ and $1/100$ of that of Algorithm~\ref{alg:ExactAlgorithm}. 

Algorithm~\ref{alg:HeuristicAlgorithm} produced the same output as Algorithm~\ref{alg:ExactAlgorithm} in approximately 70\% of instances, even near the execution limit of Algorithm~\ref{alg:ExactAlgorithm} (around $r = 30$). In practice, we observed that Algorithm~\ref{alg:HeuristicAlgorithm} rarely produced solutions that were substantially worse than optimal, and generally provided good approximations. The full set of results, including the output values of both algorithms for all samples, is available in our GitHub repository.

\begin{figure}[htbp]
  \centering
  \includegraphics[width=0.95\textwidth]{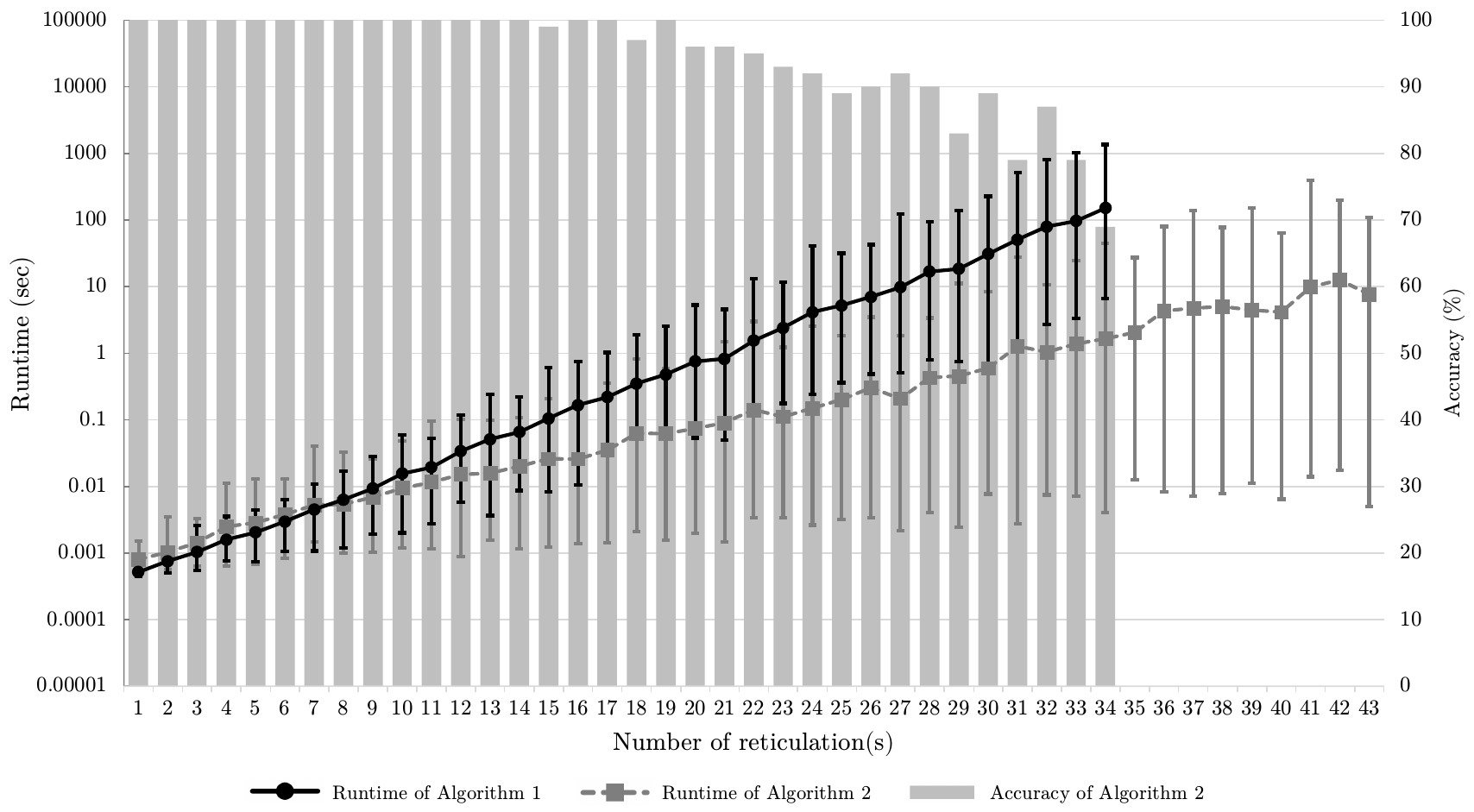}  
  \caption{
Performance of Algorithms~\ref{alg:ExactAlgorithm} and~\ref{alg:HeuristicAlgorithm} for phylogenetic networks $N$ with $n = 8$ leaves and $1 \leq r \leq 43$ reticulations. 
Runtimes (left axis, log scale) show the average over 100 samples, with error bars indicating the minimum and maximum. 
The gray bars and right axis show the percentage of instances in which the output of Algorithm~\ref{alg:HeuristicAlgorithm} exactly matched that of Algorithm~\ref{alg:ExactAlgorithm}.
 \label{fig:experiment results}}
\end{figure}

\section{Conclusion and future directions}\label{sec:conclusion}
In this paper, we have extended Hayamizu's structure theorem for rooted binary phylogenetic networks to develop a theoretical foundation for support networks. The extension from support trees allows us to analyze not only tree-based networks but also general, non-tree-based rooted phylogenetic networks. Our main contributions are threefold:

First, we established direct-product characterizations of three families of support networks for a given network $N$: all support networks $\mathcal{A}_N$, minimal support networks $\mathcal{B}_N$, and minimum support networks $\mathcal{C}_N$. These characterizations yielded closed-form counting formulas that can be computed in linear time, revealing interesting connections to well-known integer sequences such as Fibonacci, Lucas, Padovan, and Perrin numbers. We also described and implemented a linear-delay algorithm for listing the support networks in each family. 

Second, we developed a linear-time algorithm for  \textsc{Reticulation Minimization}, which finds a support network with the minimum reticulation number (minimum tier). We proved that an optimal solution can be found simply by picking any element from the set $\mathcal{C}_N$.

Third, we proposed both exact and heuristic algorithms for \textsc{Level Minimization}. While both algorithms have exponential time complexity, our experimental evaluations demonstrated their practicality across a wide range of reticulation levels. Notably, the heuristic method found optimal solutions in most cases, and when it did not, it still produced good approximations. 

There are some promising avenues for future research. Although we have discussed binary networks in this paper, we note that all results presented here apply to any non-binary network $N$ such that each vertex $v$ satisfies $\text{indeg}_N(v) \leq 2$ and $\text{outdeg}_N(v) \leq 2$. Such networks are called ``almost-binary''  \cite{SuzukiEtAl--BridgingDeviationIndices}. However, extending this framework to fully non-binary networks is challenging because the uniqueness of the maximal zig-zag trail decomposition of $N$, a key component of Hayamizu's structure theorem that underpins our work, is not guaranteed beyond the almost-binary case. It is also important to explore whether support networks with the minimum tier or level can gain meaningful insights into reticulate evolutionary processes from biological data. As our heuristic algorithm empirically produces good approximate solutions, theoretical analysis of its approximation ratio and the design of improved approximation algorithms also remain valuable pursuits.

\bibliography{takatora-counting.bib} 

\appendix

\section{Appendix}
In what follows, $N$ is a rooted binary phylogenetic network, $\mathcal{Z}=\{Z_1, \dots, Z_d\}$ is  the maximal zig-zag trail decomposition of $N$, and $\vec{\mathcal{Z}} = (Z_1, \dots, Z_d)$ is any  ordered set of the elements of $\mathcal{Z}$.  
In addition, for each $Z_i\in \mathcal{Z}$,   $m_i := |E(Z_i)|$ and we write $Z_i= (e_1, \dots, e_{m_i})$. 
\subsection{Proof of Theorem \ref{thm:fibonacci.lucas}}\label{apd.proof:fibonacci.lucas}

    We prove that $|\mathcal{S}_\mathcal{A}(Z_i)| = F_{m_i}$ holds if $Z_i$ is a fence and that $|\mathcal{S}_\mathcal{A}(Z_i)| = L_{m_i}$ holds if $Z_i$ is a crown.
    Consider a fence  $Z_i = (e_1, \dots, e_{m_i})$. As the case of $m_i\leq 2$ is trivial, we may assume $m_i\geq 3$.  
   Let $Z_i^\prime = (e_1^\prime,\dots, e_{m_i}^\prime)$ be the undirected graph obtained from $Z_i$ by ignoring all edge orientations, and let  $Z_i^{\prime\prime}  = (e_2^\prime,\dots, e_{m_i-1}^\prime)$ be its subgraph induced by the non-terminal edges of $Z_i^\prime$. 
   Since $N$ is binary, the fence $Z_i$ visits each internal vertex exactly once, so $Z_i^{\prime\prime}$ is a path (note that if we consider  almost-binary $N$ mentioned in Section \ref{sec:conclusion}, $Z_i^{\prime\prime}$ may not be a path; however, without loss of generality we may still assume that $Z_i^{\prime\prime}$ is a path because this does not affect $\mathcal{S}_\mathcal{A}(Z_i)$). 
    Let $S$ be a subset of $E(Z_i)$ and let $\overline{S} := E(Z_i)\setminus S$. 
   By  Lemma \ref{thm:bijection.ABC} and \eqref{eq:sequence for support network}, $S$ is $\mathcal{A}$-admissible if and only if 
    $e_1, e_2\notin \overline{S}$  and for any $k\in [1, m_i - 1]$,  $\{e_k, e_{k+1}\} \not \subseteq \overline{S}$.  
       Then,  there exists a bijection between $\mathcal{S}_{\mathcal{A}} (Z_i)$ and the family of matchings in the path $Z_i^{\prime\prime}$.  
    Theorem 1 in~\cite{farrell1986occurrences} says that the number of (possibly empty) matchings in a path with $\ell \geq 1$ edges equals $F_{\ell+2}$. Since $|E(Z_i^{\prime\prime})|=m_i -2$, the number of  matchings  in $Z_i^{\prime\prime}$ equals $F_{m_i}$.
    Hence, $|\mathcal{S}_\mathcal{A}(Z_i)| = F_{m_i}$ holds.

Consider a crown  $Z_i = (e_1, \dots, e_{m_i})$. Since a crown has no terminal edges, $S\subseteq E(Z_i)$  is $\mathcal{A}$-admissible if and only if for any $k\in [1, m_i - 1]$,  $\{e_k, e_{k+1}\} \not \subseteq \overline{S}$. 
Let $Z_i^\prime$ be the undirected cycle obtained from $Z_i$ by ignoring all edge orientations. Similarly to above,  there is a bijection between $\mathcal{S}_{\mathcal{A}} (Z_i)$ and the family of matchings in  the cycle $Z_i^\prime$. Theorem 2 in~\cite{farrell1986occurrences} says that the number of (possibly empty) matchings in a cycle with $\ell\geq 3$ edges equals $L_\ell$.  Hence, $|\mathcal{S}_\mathcal{A}(Z_i)| = L_{m_i}$ holds. 

Theorems 5.6 and 5.8 in \cite{koshy2019fibonacci} say that $F_n$ and $L_n$ are expressed by the closed-forms $F_n = (\phi^n -(-\phi)^{-n})/\sqrt{5}$ and $L_n = \phi^n + (1-\phi)^n$, respectively. Then, $F_n = \Theta(\phi^n)$ and $L_n = \Theta(\phi^n)$ hold. This implies $\Theta(|\mathcal{A}_N|) = \Theta(\phi^{|E(Z_1)|}\times \dots \times \phi^{|E(Z_d)|}) = \Theta(\phi^{|E(N)|})$ holds. 
    This completes the proof.

      \subsection{Proof of Theorem \ref{thm:counting by padovan, perrin}}\label{apd.proof.B.padovan.perrin}
    We prove that $|\mathcal{S}_\mathcal{B}(Z_i)| = P_{m_i}$ holds if $Z_i$ is a fence and $|\mathcal{S}_\mathcal{B}(Z_i)| = Q_{m_i}$ holds if $Z_i$ is a crown.
    Consider a fence $Z_i = (e_1, \dots, e_{m_i})$. As the case of $m_i\leq 2$ is trivial, we may assume $m_i\geq 3$.   Let $Z_i^{\prime\prime}$ be the undirected path as in the proof of Theorem \ref{thm:fibonacci.lucas}.  
    An $\mathcal{A}$-admissible subset $S$ of $E(Z_i)$ is minimal  (namely, $\mathcal{B}$-admissible) if and only if a matching in $Z_i^{\prime\prime}$ corresponding to $\overline{S}:=E(Z_i)\setminus S$ is maximal. 
   Then, there exists a bijection between $\mathcal{S}_{\mathcal{B}}(Z_i)$ and the family of maximal matchings in $Z_i^{\prime\prime}$. 
   Proposition 3.1 in \cite{padovan-perrin} say that the number of maximal matchings in a path with $\ell$ edges equals $P_{\ell+2}$. 
   Since $|E(Z_i^{\prime\prime})|=m_i -2$, the number of maximal matchings  in $Z_i^{\prime\prime}$ equals $P_{m_i}$.
    Hence, $|\mathcal{S}_\mathcal{B}(Z_i)| = P_{m_i}$ holds. 
   

    Consider a crown $Z_i = (e_1, \dots, e_{m_i})$. Let $Z_i^\prime$ be the undirected cycle as in the proof of Theorem \ref{thm:fibonacci.lucas}. Then,  
    there exists a bijection between $\mathcal{S}_{\mathcal{B}}$ and the family of maximal matchings in the cycle $Z_i^\prime$. 
    By the statement after Proposition 3.9 in \cite{padovan-perrin},  the number of maximal matchings in a cycle with $\ell\geq 3$ edges equals $Q_\ell$. 
    Thus, $|\mathcal{S}_\mathcal{B}(Z_i)| = Q_{m_i}$. 
    
    By \cite{padovan-perrin}, $P_n = \Theta(\psi^n)$ and $Q_n = \Theta(\psi^n)$ hold.  Hence, $\Theta(|\mathcal{B}_N|) = \Theta(\psi^{|E(Z_1)|}\times \dots \times \psi^{|E(Z_d)|}) = \Theta(\psi^{|E(N)|})$ holds. This completes the proof.


\subsection{Proof of Theorem \ref{thm:alg1}}
\label{apd:exact correctness}
%
As Algorithm \ref{alg:ExactAlgorithm} checks the level of each support network in $\mathcal{B}_N$, our goal is to show that $\mathcal{B}_N$ contains a level-$k$ support network of $N$. To obtain a contradiction, assume that $\mathrm{level}(G) > k$ holds for any $G \in \mathcal{B}_N$. This implies that any level-$k$ support network $G^\prime$ of $N$ is in $\mathcal{A}_N \setminus \mathcal{B}_N$. 
Since $G^\prime$ is not minimal, there exists a minimal one $G\in\mathcal{B}_N$ with $G \subsetneq G^\prime$. 
Then,  $\mathrm{level}(G) \leq \mathrm{level}(G^\prime) = k$ since each block of $G$ is a subgraph of its corresponding block of $G^\prime$, which is a contradiction. This completes the proof of the correctness of Algorithm \ref{alg:ExactAlgorithm}. 

Since $N$ is a binary network, we have $O(|V(N)|) = O(|E(N)|)$. Also, as any minimal support network $G \in \mathcal{B}_N$ is a subgraph of $N$, we have $O(|E(G)|) = O(|E(N)|)$.  
For each $G \in \mathcal{B}_N$, it takes $O(|V(G)| + |E(G)|) = O(|E(N)|)$ time to compute  its block decomposition $\{G_1, \dots, G_k\}$ using depth-first search \cite{HopcroftTarjan-1973-Algorithm447Efficient}.   
  For any block $G_i$,  one can compute  $r_i = |E(G_i)|-|V(G_i)|+1$ in $O(|E(G_i)|)$ time.  Then, for each $G \in \mathcal{B}_N$, it takes $O(|E(G_1)|+\dots +|E(G_k)|)= O(|E(G)|)=O(|E(N)|)$ time to calculate  $\mathrm{level}(G)$, namely the maximum  $r_i$ across all blocks $G_i$ of $G$. Overall,  it takes $O(|\mathcal{B}_N|\cdot |E(N)|)$ time to obtain the minimum value of $\mathrm{level}(G)$ across all $G\in \mathcal{B}_N$. By Theorem~\ref{thm:counting by padovan, perrin}, $O(|\mathcal{B}_N|\cdot |E(N)|) = O(\psi^{|E(N)|}\cdot |E(N)|)$ holds.



\end{document}